\documentclass[letterpaper, 10 pt, conference]{ieeeconf}  

\IEEEoverridecommandlockouts                              

\usepackage{amsmath,amssymb,bm,bbm,mathrsfs,amscd}
\usepackage{color}
\usepackage{dsfont}
\usepackage{graphicx}
\usepackage{epsfig}
\usepackage{subfigure}
\usepackage{algorithm}
\usepackage[noend]{algpseudocode}
\usepackage{tikz}

\newtheorem{theorem}{Theorem}
\newtheorem{definition}{Definition}

\newtheorem{lemma}{Lemma}

\newtheorem{example}{Example}
\newtheorem{remark}{Remark}
\newtheorem{assumption}{Assumption}

\newcommand{\RR}{{\mathbb{R}}}
\newcommand{\NN}{{\mathbb{N}}}
\newcommand{\EE}{{\mathbb{E}}}
\newcommand{\PP}{{\mathbb{P}}}
\newcommand{\JJ}{{\mathbb{J}}}
\newcommand{\FF}{{\mathbb{F}}}

\newcommand{\mc}{\mathcal}
\newcommand{\norm}[1]{\left\|#1\right\|}

\newcommand{\EEk}[1]{\EE\left[#1|\mc F_k\right]}

\newcommand{\op}{\operatorname}

\newcommand{\bs}{\boldsymbol}
\newcommand{\fineass}{\hfill\small$\blacksquare$}

\title{\LARGE \bf
Distributed projected--reflected--gradient algorithms for stochastic generalized Nash equilibrium problems
}

\author{Barbara Franci and Sergio Grammatico
\thanks{The authors are with the Delft Center for System and Control, TU Delft, The Netherlands
        {\tt\footnotesize \{b.franci-1, s.grammatico\}@tudelft.nl}}%
\thanks{This work was partially supported by NWO under research projects OMEGA (613.001.702) and P2P-TALES (647.003.003), and by the ERC under research project COSMOS (802348).}}

\begin{document}

\maketitle
\thispagestyle{empty}
\pagestyle{empty}

\begin{abstract}
We consider the stochastic generalized Nash equilibrium problem (SGNEP) with joint feasibility constraints and expected--value cost functions. We propose a distributed stochastic projected reflected gradient algorithm and show its almost sure convergence when the pseudogradient mapping is monotone and the solution is unique. The algorithm is based on monotone operator splitting methods tailored for SGNEPs when the expected-value pseudogradient mapping is approximated at each iteration via an increasing number of samples of the random variable.
Finally, we show that a preconditioned variant of our proposed algorithm has convergence guarantees when the pseudogradient mapping is cocoercive. 
\end{abstract}

\section{Introduction}

Stochastic generalized Nash equilibrium problems (SGNEPs) have received some attention from the system and control community \cite{koshal2013,yu2017,franci2020tac}. In a SGNEP, a set of agents interacts with the aim of minimizing their expected-value cost functions while subject to some joint feasibility constraints. Moreover, both the cost function and the constraints depend on the strategies of the other agents. SGNEPs arise when there is some uncertainty, modeled via a random variable with an unknown distribution. The interest on the topic is related to their possible applications in networked Cournot game with market capacity constraints and uncertain demand \cite{abada2013} as transportation systems, where the drivers perception of travel time is uncertain \cite{watling2006}, and electricity markets where companies schedule their production without knowing the demand \cite{henrion2007}.

When the random variable is known, the expected--value formulation can be solved via a standard technique for deterministic GNEPs \cite{yi2019,belgioioso2018}.
In fact, one possible approach is to recast the problem as a stochastic variational inequality (SVI) through the use of the Karush--Kuhn--Tucker conditions. Then, the problem can be rewritten as a monotone inclusion and solved via operator splitting techniques. 
The difficulty in the stochastic case is that the pseudogradient is usually not directly accessible, for instance because the expected value is hard to compute. For this reason, in many situations, the search for a solution of a SVI relies on samples of the random variable.
One of the most used approximation schemes in the literature is the stochastic approximation (SA). 
In this case, the approximation is done by using only one (or a finite number of) realization of the random variable. SA was first presented in \cite{robbins1951}, it is not computationally expensive even if sometimes it requires strong assumptions on the pseudogradient mapping and on the parameters \cite{koshal2013,cui2016}. 

Depending on the monotonicity assumptions on the pseudogradient mapping or the affordable computational complexity, there are different algorithms that can be used to find a SGNE. Among others, one can consider the stochastic preconditioned forward--backward (SpFB) algorithm \cite{franci2020tac} which is guaranteed to converge under cocoercivity of the pseudogradient and by demanding one projection step per iteration. However, cocoercivity is not the weakest possible assumption, hence one would like to have an algorithm that converges under mere monotonicity. Under this assumption, the stochastic forward--backward--forward (SFBF) algorithm involves only one projection step per iteration but two costly evaluation of the pseudogradient mapping \cite{staudigl2019}. Another alternative is the stochastic extragradient (SEG) algorithm whose iterates are characterized by two projection steps and two evaluation of the pseudogradient mapping \cite{iusem2017,cui2019}. To summarize, having weaker assumptions comes at the price of implementing computationally expensive algorithms.

In this paper, we propose two special instances of the stochastic projected reflected gradient (SPRG) algorithm for SGNEPs. The SPRG involves only one projection step and one evaluation of the pseudogradient, computed in the reflection of the second-last iterate on the last one, i.e., it uses the previous two iterates for the updates.
Specifically, our contributions are the following.
\begin{itemize}
\item We exploit two splitting techniques to obtain two distributed algorithms, with and without preconditioning, that are instances of the SPRG scheme.
\item We prove convergence of the first algorithm to a SGNE if the pseudogradient mapping is monotone and the solution is unique (Section \ref{sec_conv_uni}), differently from \cite{cui2019} where it is assumed to be monotone and ``weak sharp''.
\item The weak sharpness property, discussed in Section \ref{sec_coco}, is hard to check on the problem data but it is guaranteed to hold on a cocoercive operator. Therefore, we prove convergence of the second, preconditioned algorithm when the pseudogradient mapping is cocoercive (Section \ref{sec_conv_coco}).
\end{itemize}
We note that PRG was first presented for deterministic VI in \cite{malitsky2015} and for SVI in \cite{cui2016,cui2019}.
Nevertheless, this is the first time that it is designed for SGNEPs, i.e. a non-trivial instance of SVI with coupling constraints between the agents. Remarkably, under uniqueness of the solution, our algorithm has convergence guarantees also for merely monotone pseudogradient mappings. This is a significant advantage compared to the SpFB which may not converge in that case.

\textit{Notation and preliminaries:} We use the same notation as \cite{franci2020tac,bau2011,facchinei2007} and we recall that 
for all $x,y\in C$, a mapping $F:C\subseteq\RR^n\to\RR^n$ is \cite{bau2011,facchinei2007}: \textit{pseudomonotone} if $(x-y)^{T} F(y) \geq 0 \Rightarrow(x-y)^{T} F(x) \geq 0$; \textit{monotone} if $(F(x)-F(y))^{T}(x-y) \geq 0$; \textit{maximally monotone} if there exists no monotone operator $G :C\to\RR^n$ such that the graph of $G$ properly contains the graph of $F$; $\beta$-\textit{cocoercive} if for $\beta>0$ $(F(x)-F(y))^{T}(x-y) \geq \beta\|F(x)-F(y)\|^{2}$; $L$-\textit{Lipschitz continuous} if for $L>0$ $\norm{F(x)-F(y)} \leq L\norm{x-y}$. 
\vspace{-.1cm}
\section{Stochastic generalized Nash equilibrium problem}
We consider a set of agents $\mc I=\{1,\dots,N\}$, each of them choosing its decision variable $x_i\in\RR^{n_i}$ from its local decision set $\Omega_i\subseteq\RR^{n_i}$. The aim of each agent is to minimize its local cost function $\JJ_i: \RR^{n} \rightarrow \RR$ defined as
$$\JJ_i(x_i,\boldsymbol x_{-i}):=\EE_\xi[J_i(x_i,\boldsymbol x_{-i},\xi(\omega))]$$
for some measurable function $J_i:\mc \RR^{n}\times \RR^d\to \RR$ and $n=\sum_{i=1}^N n_i$. We note that the cost function depends on the local variable $x_i$, on the decisions of the other agents $\bs x_{-i}=\op{col}((x_j)_{j \neq i})$ and on the random variable $\xi:\Xi\to\RR^d$ which expresses the uncertainty in the cost function. Given the associated probability space $(\Xi, \mc F, \PP)$, $\EE_\xi$ represents the mathematical expectation with respect to the distribution of $\xi$\footnote{For brevity, we use $\xi$ instead of $\xi(\omega)$, $\omega\in\Xi$, and $\EE$ instead of $\EE_\xi$.}. We assume that $\EE[J_i(\boldsymbol x,\xi)]$ is well defined for all the feasible $\boldsymbol x=\op{col}(x_1,\dots,x_N)$. 

The cost function should satisfy some assumptions, postulated next, to make our analysis possible. Such assumptions are standard for (stochastic) Nash equilibrium problems \cite{ravat2011}.
\begin{assumption}\label{ass_J}
For each $i \in \mc I$ and $\boldsymbol{x}_{-i} \in \mc{X}_{-i}$ $\JJ_{i}(\cdot, \boldsymbol{x}_{-i})$ is convex and continuously differentiable. \\
For each $i\in\mc I$ and $\xi \in \Xi$, $J_{i}(\cdot,\boldsymbol x_{-i},\xi)$ is convex, Lipschitz continuous and continuously differentiable. For each $\boldsymbol x_{-i}$, $J_{i}(x_i,\bs x_{-i},\cdot)$ is measurable and Lipschitz continuous with constant $\ell_i(\boldsymbol x_{-i},\xi)$, integrable in $\xi .$ \fineass
\end{assumption}

Since we consider a SGNEP, we introduce affine shared constraints of the form $A\bs x\leq b$, where $A=\left[A_{1}, \ldots, A_{N}\right]\in\RR^{m\times n}$, $A_i \in \RR^{m \times n_i}$ indicates how agent $i$ is involved in the coupling constraints and $b\in\RR^m$. Thus, we denote each agent $i \in \mc I$  feasible decision set with the set-valued mapping 
\begin{equation}\label{eq_constr}
\mc{X}_i(\boldsymbol x_{-i}) :=\{y_i \in \Omega_i\; | \;\textstyle{A_i y_i \leq b-\sum_{j \neq i}^{N} A_j x_j}\},
\end{equation}
and the collective feasible set as
\begin{equation}\label{collective_set}
\bs{\mc{X}}=\bs\Omega \cap\left\{\boldsymbol{y} \in \RR^{n} | A\boldsymbol{y}-b \leq {\bf{0}}_{m}\right\}
\end{equation}
where $\bs\Omega=\prod_{i=1}^N\Omega_i$.
We suppose that there is no uncertainty in the constraints.
Standard assumptions for the constraints sets are postulated next \cite{facchinei2007vi}.

\begin{assumption}\label{ass_X}
For each $i \in \mc I,$ $\Omega_{i}$ is nonempty, compact and convex.
$\bs{\mc{X}}$ satisfies Slater's constraint qualification. \fineass
\end{assumption}

The aim of each agent $i$, given the decision variables of the other agents $\bs x_{-i}$, is to choose a decision $x_i$, that solves its local optimization problem, i.e.,
\begin{equation}\label{eq_game}
\forall i \in \mc I:\quad\left\{\begin{array}{cl}
\min\nolimits_{x_i \in \Omega_i} & \JJ_i(x_i,\boldsymbol x_{-i}) \\ 
\text { s.t. } & A_i x_i \leq b-\sum_{j \neq i}^{N} A_j x_j.
\end{array}\right.
\end{equation}
From a game-theoretic perspective, we aim at computing a stochastic generalized Nash equilibrium (SGNE) \cite{ravat2011}.
\begin{definition}
A collective variable $\boldsymbol{x}^*\in\bs{\mc X}$ is a stochastic generalized Nash equilibrium if, for all $i \in \mc I$:
$$\JJ_i(x_i^{*}, \boldsymbol x_{-i}^{*}) \leq \inf \{\JJ_i(y, \boldsymbol x_{-i}^{*})\; | \; y \in \mc{X}_i(\boldsymbol x_{-i}^{*})\}.$$
\end{definition}

When Assumptions \ref{ass_J}-\ref{ass_X} hold, existence of a SGNE of the game in \eqref{eq_game} is guaranteed by \cite[\S 3.1]{ravat2011}.

Among all the possible Nash equilibria, we focus on those that are also solutions of an associated (stochastic) variational inequality. 
First let us define the pseudogradient mapping as 
\begin{equation}\label{eq_grad}
\FF(\boldsymbol x)=\op{col}\left(\EE[\nabla_{x_{i}} J_{i}(x_{i}, \boldsymbol x_{-i},\xi_i)]_{i\in\mc I}\right),
\end{equation} 
where the exchange between the expected value and the gradient is possible because of Assumption \ref{ass_J} \cite[Lem. 3.4]{ravat2011}. 
Then, the stochastic variational inequality $\op{SVI}(\bs{\mc X},\FF)$ is the problem of finding $\bs x^{*} \in \bs{\mc X}$ such that 
\begin{equation}\label{eq_vi}
\langle \FF(\bs x^*),\bs x-\bs x^*\rangle\geq 0, \;\text { for all } \bs x \in \bs{\mc X}.
\end{equation}
with $\FF(\bs x)$ as in \eqref{eq_grad} and $\bs{\mc X}$ as in \eqref{eq_constr}.
We also note that any solution of $\op{SVI}(\bs{\mc X} , \FF)$ is a SGNE of the game in (\ref{eq_game}) while the opposite does not hold in general. 
\begin{remark}
Under Assumptions \ref{ass_J}-\ref{ass_X}, the solution set of $\op{SVI}(\bs{\mc X},\FF)$ is non empty and compact, i.e. $\op{SOL}(\bs{\mc X},\FF)\neq\varnothing$ \cite[Corollary 2.2.5]{facchinei2007}.
\fineass\end{remark}

\vspace{-.1cm}
\subsection{Operator-theoretic characterization}
In this section, we rewrite the SGNEP as a monotone inclusion, i.e., the problem of finding a zero of a set-valued monotone operator.
To this aim, we characterize the SGNE of the game in terms of the Karush--Kuhn--Tucker (KKT) conditions for the coupled optimization problems in \eqref{eq_game}. Let us denote with $\lambda_i \in \RR_{ \geq 0}^{m}$ the dual variable associated with the coupling constraints. Then, the associated KKT optimality conditions reads as \cite{facchinei2007}
\begin{equation}\label{game_kkt}
\forall i \in \mc I: \left\{\begin{array}{l}
0 \in \EE[\nabla_{x_i} J_i(x_i^{*}, \boldsymbol x_{-i}^{*},\xi)]+\mathrm{N}_{\Omega_i}\left(x_i^{*}\right)+A_i^{\top} \lambda_i \\
0\in -(A\boldsymbol x^*-b)+\mathrm N_{\RR^m_{\geq 0}}(\lambda^*).
\end{array}\right.
\end{equation}
%
%
From \cite[Th. 3.1]{facchinei2007vi}, it follows that $\boldsymbol x^*$ is a solution of the SGNEP if the KKT conditions hold with
$\lambda_1=\lambda_2=\dots=\lambda_N=\lambda^*$, i.e., $\boldsymbol x^*$ is a solution of $\op{SVI} (\bs{\mc X}, \FF)$ in \eqref{eq_vi}.
In words, \cite[Th. 3.1]{facchinei2007vi} says that variational equilibria are those such that the shared constraints have the same dual variable $\lambda^*$ for all the agents. 
We conclude this section writing the KKT conditions as 
\begin{equation}\label{eq_T}
0\in\mc T(\bs x,\bs \lambda):=\left[\begin{array}{c}
\mathrm{N}_{\Omega}(\boldsymbol x)+\FF(\boldsymbol x)+A^{\top}\bs \lambda \\ 
\mathrm{N}_{\RR_{ \geq 0}^{m}}(\bs\lambda)-(A \boldsymbol x-b)
\end{array}\right], 
\end{equation}
where $\mc T:\bs{\mc X}\times \RR^m_{\geq 0}\rightrightarrows \RR^{n}\times\RR^m$ is a set-valued mapping. This notation in helpful to obtain the algorithms and for the convergence analysis.
\vspace{-.1cm}
\subsection{Approximation scheme}
Since the distribution of the random variable is unknown, we replace the expected value with an approximation. For the convergence analysis, we use the stochastic approximation (SA) scheme. 
We assume to have access to a pool of i.i.d. samples of the random variable collected, for all $k\in\NN$ and for each agent $i\in\mc I$, in the vectors $\bar\xi_i^k=\op{col}(\xi_i^{(1)},\dots,\xi_i^{(N_k)})$. At each time $k$, the approximation is 
\begin{equation}\label{eq_F_SA}
F^{\textup{SA}}_i(x_i^k,\boldsymbol x_{-i}^k,\bar\xi_i^k)=\textstyle{\frac{1}{N_k}\sum_{s=1}^{N_k}\nabla_{x_i} J_i(x_i^k,\boldsymbol x_{-i}^k,\xi_i^{(s)})}
\end{equation}
where $N_k$ is the batch size, i.e., the number of sample to be taken. 
We define the distance between the expected value and its approximation as 
$$\epsilon^k=F^{\textup{SA}}(\boldsymbol x^k,\bs\xi^k)-\FF(\boldsymbol x^k),$$
where $F^\textup{SA}(\bs x,\bs\xi)=\op{col}(F_i^{\textup{SA}}(\bs x,\bar\xi_i))$ and $\bs \xi=\op{col}(\bar \xi_1,\dots,\bar\xi_n)$.
From now on, the superscript SA indicates that we use $F^\textup{SA}$ as in \eqref{eq_F_SA}.
Let us introduce the filtration $\mc F=\{\mc F_k\}$, i.e., a family of $\sigma$-algebras such that $\mathcal{F}_{0} = \sigma\left(X_{0}\right)$, for all $k \geq 1$, 
$\mathcal{F}_{k} = \sigma\left(X_{0}, \xi_{1}, \xi_{2}, \ldots, \xi_{k}\right)$
and $\mc F_k\subseteq\mc F_{k+1}$ for all $k\geq0$. The filtration $\mc F$ collects the informations that each agent has at the beginning of each iteration $k$. We note that the process $\epsilon_k$ is adapted to $\mc F_k$ and it satisfies the following assumption \cite{staudigl2019,iusem2017}.
\begin{assumption}\label{ass_error}
For al $k\geq 0$, $\EEk{\epsilon_k}=0$ a.s..
\fineass\end{assumption}
Moreover, the stochastic error has a vanishing second moment that  depends on the increasing number of samples $N_k$ taken at each iteration.
\begin{assumption}\label{ass_batch}
There exist $c,k_0,a>0$ such that, for all $k\in\NN$, 
$N_k\geq c(k+k_0)^{a+1}.$
For all $k$ and $C>0$, the stochastic error is such that 
\begin{equation}\label{eq_var}
\textstyle{\EE[\|\epsilon^k\||\mc F_k]\leq \frac{C\sigma^2}{N_k}\text{ a.s..}}
\end{equation}
\end{assumption}
The bound for the stochastic error in \eqref{eq_var} can be obtained as a consequence of some milder assumptions and it is called variance reduction; we refer to \cite[Lem. 3.12]{iusem2017}, \cite[Lem. 6]{franci2020tac} for more details. Concerning the batch size, Assumption \ref{ass_batch} is standard in the SA literature \cite[Eq. 11]{iusem2017}, \cite[Eq. v-SPRG]{cui2019} since it implies that the sequence $(1/N_k)_{k\in\NN}$ is summable.
\vspace{-.1cm}
\begin{algorithm}[t]
\caption{Distributed stochastic projected reflected gradient (SPRG)}\label{algo_no}
Initialization: $x_i^0 \in \Omega_i, \lambda_i^0 \in \RR_{\geq0}^{m},$ and $z_i^0 \in \RR^{m} .$\\
Iteration $k$: Agent $i$\\
(1) Updates 
$$\begin{aligned}
&\tilde x_{i}^{k}=2x_i^k-x_i^{k-1}\\
&\tilde z_{i}^{k}=2z_i^k-z_i^{k-1}\\
&\tilde \lambda_{i}^{k}=2\lambda_i^k-\lambda_i^{k-1}\\
\end{aligned}$$
(2) Receives $x_{j}^k$, $j \in \mathcal{N}_{i}^{J}$ and $z_j^k, \lambda_{j}^k$, $j \in \mathcal{N}_{i}^{\lambda}$, then updates:
$$\begin{aligned}
x_i^{k+1}&=\op{proj}_{\Omega_{i}}[ x_i^k-\alpha_{i}(F^\textup{SA}_{i}(\tilde x_i^k, \tilde {\boldsymbol{x}}_{-i}^k,\xi_i^k)+A_{i}^\top \tilde \lambda_i^k)]\\
z_i^{k+1}&= z_i^k-\nu_{i} \textstyle{\sum_{j \in \mathcal{N}_{i}^{\lambda}} w_{i,j}(\tilde \lambda_i^k-\tilde \lambda_{j}^k)}\\
\lambda_i^{k+1}&=\op{proj}_{\RR^m_{\geq 0}}\{ \lambda_i^k+\tau_{i}(A_{i}\tilde x_i^k-b_{i})\\
&-\tau\textstyle{\sum_{j \in \mathcal{N}_{i}^{\lambda}} w_{i,j}[(\tilde z_{i}^{k}-\tilde z_j^k)-(\tilde \lambda_i^k-\tilde \lambda_j^k)]\}}
\end{aligned}$$\vspace{-.2cm}
\end{algorithm}
\section{Distributed stochastic projected reflected gradient algorithms}\label{sec_dist}
In this section, we propose two distributed instances of a stochastic projected reflected gradient (SPRG) algorithm for finding a v-SGNE of the game in \eqref{eq_game}. The iterations are presented in Algorithms \ref{algo_no} and \ref{algo_si} and are inspired by \cite{cui2019,malitsky2015}.
We note that in Algorithm \ref{algo_no} only one communication round is needed for the update of the local variables $x_i^k$, $z_i^k$ and $\lambda_i^k$ at each iteration $k$, while the preconditioned version requires an extra exchange of information before the update of the dual variable. However, in Algorithm \ref{algo_si} the reflection step is not done for the auxiliary variable. 
Since we want the algorithm to be distributed, we assume that each agent $i$ only knows its local data $\Omega_i$, $A_i$ and $b_i$. Moreover, each player is able to compute the approximation $F^\textup{SA}(\bs x,\bs \xi)$ in \eqref{eq_F_SA} of $\FF(\bs x)$ in \eqref{eq_grad}, given the collective decision $\bs x$. We assume that each agent has access to all the decision variables that affect its pseudogradient (full decision information setup). These information are collected in the sets $\mc N_i^J$, $i\in\mc I$, i.e., the set of agents $j$ whose decision $x_j$ explicitly influences $J_i$.
Since the v-SGNE requires consensus of the dual variables, we introduce an auxiliary variable $z_i\in\RR^m$, $i\in\mc I$. The role of $\bs z=\op{col}(z_1,\dots,z_N)$ is to help reaching consensus and, along with a local copy of the dual variable $\lambda_i$, it is shared through the graph $\mc G^\lambda=(\mc I,\mc E^\lambda)$. The set of edges $\mc E^\lambda$ represents the exchange of the private information on the dual variables: $(i,j)\in\mc E^\lambda$ if agent $i$ can receive $\{\lambda_j,z_j\}$ from agent $j$. The set of neighbors of $i$ in $\mc G^\lambda$ is indicated with $\mathcal{N}_{i}^{\lambda}=\{j |(j, i) \in \mathcal{E}_{\lambda}\}$ \cite{yi2019,yu2017}. 
Since each agent feasible set implicitly depends on all the other agents decisions, to reach consensus of the dual variables, all agents must coordinate and therefore, $\mc G^\lambda$ must be connected.
\begin{assumption}\label{ass_graph}
The dual-variable communication graph $\mc G^\lambda$ is undirected and connected.\fineass
\end{assumption}

The weighted adjacency matrix of the dual variables graph is indicated with $W\in\RR^{N\times N}$. 
Let $L=D-W\in\RR^{N\times N}$ be the Laplacian matrix associated to $W$, where $D=\op{diag}(d_1,\dots,d_N)$ is the diagonal matrix of the agents degrees $d_i=\sum_{j=1}^{N} w_{i,j}$. It follows from Assumption \ref{ass_graph} that the adjacency matrix $W$ and the Laplacian $L$ are both symmetric, i.e., $W=W^\top$ and $L = L^\top$. 
\begin{algorithm}[t]
\caption{Distributed stochastic preconditioned projected reflected gradient (SpPRG)}\label{algo_si}
Initialization: $x_i^0 \in \Omega_i, \lambda_i^0 \in \RR_{\geq0}^{m},$ and $z_i^0 \in \RR^{m} .$\\
Iteration $k$: Agent $i$\\
(1) Updates
$$\begin{aligned}
&\tilde x_{i}^{k}=2x_i^k-x_i^{k-1}\\
&\tilde \lambda_{i}^{k}=2\lambda_i^k-\lambda_i^{k-1}\\
\end{aligned}$$
(2) Receives $\tilde x_{j}^{k}$, $j \in \mathcal{N}_{i}^{J}$ and $\lambda_{j}^k$, $j \in \mathcal{N}_{i}^{\lambda}$ then updates:
$$\begin{aligned}
&x_i^{k+1}=\op{proj}_{\Omega_{i}}[x_i^k-\alpha_{i}(F^\textup{SA}_{i}(\tilde x_i^k, \tilde {\boldsymbol{x}}_{-i}^k,\xi_i^k)-A_{i}^\top \lambda_i^k)]\\
&z_i^{k+1}=z_i^k+v_{i} \textstyle{\sum\nolimits_{j \in \mathcal{N}_{i}^{\lambda}} w_{i,j}(\lambda_i^k-\lambda_{j}^k)}\\
\end{aligned}$$
(3) Receives $\tilde x_{j}^{k+1}$, $j \in \mathcal{N}_{i}^{J}$, $z_{j}^{k+1}$, $\tilde \lambda_{j}^k$, $j \in \mathcal{N}_{i}^{\lambda}$ then updates:
$$\begin{aligned}
&\lambda_i^{k+1}=\op{proj}_{\RR_{+}^{m}}\left[\lambda_i^k+\sigma_{i}\left(A_{i}(2 x_i^{k+1}-x_i^k)-b_{i}\right)\right.\\
&\qquad+\sigma_{i}\textstyle{\sum\nolimits_{j \in \mathcal{N}_{i}^{\lambda}} w_{i,j}\left(2(z_i^{k+1}-z_{j}^{k+1})-(z_i^k-z_{j}^k)\right)}\\
&\qquad-\sigma_{i}\textstyle{\sum\nolimits_{j \in \mathcal{N}_{i}^{\lambda}} \left.w_{i,j}(\tilde\lambda_i^k-
\tilde\lambda_{j}^k)\right]}\\
\end{aligned}$$\vspace{-.2cm}
\end{algorithm}

To obtain the distributed iterations presented in Algorithm \ref{algo_no} and \ref{algo_si}, we exploit a splitting technique starting from the operator $\mc T $ in \eqref{eq_T}. First, let us note that it can be written as $\mc T=\mc Q+\mc P+\mc R$ where
\begin{equation}
\begin{aligned}
\mathcal{Q}&:\left[\begin{smallmatrix}
\boldsymbol{x} \\
\boldsymbol{\lambda}
\end{smallmatrix}\right] \mapsto\left[\begin{smallmatrix}
\mathbb{F}(\boldsymbol{x}) \\
b
\end{smallmatrix}\right], \\
\mc R&:\left[\begin{smallmatrix}
\boldsymbol{x} \\
\boldsymbol{\lambda}
\end{smallmatrix}\right] \mapsto\left[\begin{smallmatrix}
0 & A^{\top} \\
-A & 0
\end{smallmatrix}\right]\left[\begin{smallmatrix}
\boldsymbol{x} \\
\boldsymbol{\lambda}
\end{smallmatrix}\right], \\
\mathcal{P}&:\left[\begin{smallmatrix}
\boldsymbol{x} \\
\boldsymbol{\lambda}
\end{smallmatrix}\right] \mapsto\left[\begin{smallmatrix}
\mathrm{N}_{\Omega}(\boldsymbol{x}) \\
\mathrm{N}_{\mathbb{R}_{\geq 0}^{m}}(\lambda)
\end{smallmatrix}\right],
\end{aligned}
\end{equation}
and $\boldsymbol\lambda=\op{col}(\lambda_1,\dots,\lambda_N)\in\RR^{Nm}$. Moreover, let ${\bf{L}}=L\otimes \op{I}_m\in\RR^{Nm\times Nm}$, ${\bf{A}}=\op{diag}\{A_1,\dots,A_N\}\in\RR^{Nm\times n}$ and let us define the column vector $\bs b$ of suitable dimensions.

Following \cite{yi2019}, to force consensus on the dual variables, we impose the Laplacian constraint ${\bf{L}}\bs\lambda=0$. Then, to preserve monotonicity, we augment the operators introducing the auxiliary variable $\bs z$. 
Exploiting the splitting $\mc T=(\mc Q+\mc R)+\mc P$, we consider the following extended operators:
\begin{equation}\label{eq_AB}
\begin{aligned}
\mc A&:\left[\begin{smallmatrix}
\bs{x} \\
\bs z\\
\bs \lambda
\end{smallmatrix}\right] \mapsto\left[\begin{smallmatrix}
\FF(\bs{x}) \\ 
0\\
{\bf{L}}\bs\lambda+\bs b
\end{smallmatrix}\right]+
\left[\begin{smallmatrix}
0 & 0 & {\bf{A}}^{\top} \\ 
0 & 0 & {\bf{L}}\\
-{\bf{A}} & -{\bf{L}} & 0
\end{smallmatrix}\right]\left[\begin{smallmatrix}
\bs{x} \\ 
\bs z\\
\bs \lambda
\end{smallmatrix}\right]\\
\mc B&:\left[\begin{smallmatrix}
\bs{x} \\ 
\bs z\\
\bs \lambda
\end{smallmatrix}\right]\mapsto\left[\begin{smallmatrix}
\op{N}_\Omega(\bs{x}) \\ 
{\bf{0}}\\
\mathrm{N}_{\RR_{ \geq 0}^{m}}(\lambda)
\end{smallmatrix}\right].
\end{aligned}
\end{equation}

Since the distribution of the random variable is unknown, we replace $\FF(\bs x)$ in \eqref{eq_grad} with $F^\textup{SA}$ in \eqref{eq_F_SA} and $\mc A$ with $\hat{\mc A}=\mc A^\textup{SA}$.
Then, given $\boldsymbol\omega=\op{col}(\boldsymbol x,\boldsymbol z,\boldsymbol \lambda)$, Algorithm \ref{algo_no} in compact form reads as the SPRG iteration
\begin{equation}\label{eq_algo2}
\bs\omega^{k+1}=(\op{Id}+\Phi^{-1}\bar{\mc B})^{-1}(\bs\omega^k-\Phi^{-1}\hat{\mc A}(2\bs\omega^k-\bs\omega^{k-1})),
\end{equation}
where $\Phi\succ0$ contains the inverse of step size sequences, i.e., 
$\Phi=\op{diag}(\alpha^{-1},\nu^{-1},\sigma^{-1}),$
and $\alpha^{-1}$, $\nu^{-1}$, $\sigma^{-1}$ are diagonal matrices.

Another possible splitting of the operator $\mc T$ in \eqref{eq_T} can be considered, namely, $\mc T=\mc Q+(\mc P+\mc R)$. Therefore, we can write a different couple of extended operators as
\begin{equation}\label{eq_CD}
\begin{aligned}
\mc C&:\left[\begin{smallmatrix}
\boldsymbol x \\
\boldsymbol z\\
\boldsymbol \lambda
\end{smallmatrix}\right] \hspace{-.1cm}\mapsto\left[\begin{smallmatrix}
\FF(\boldsymbol x) \\ 
0\\
\bs b
\end{smallmatrix}\right]+\left[\begin{smallmatrix}
0\\
0\\
{\bf{L}}\boldsymbol\lambda
\end{smallmatrix}\right]\\
\mc D &:\left[\begin{smallmatrix}
\boldsymbol x \\ 
\boldsymbol z\\
\boldsymbol \lambda
\end{smallmatrix}\right]\mapsto\left[\begin{smallmatrix}
\op{N}_{\Omega}(\boldsymbol x) \\ 
{\bf{0}}_{mN}\\
\op{N}_{\RR_{ \geq 0}^{m}}(\boldsymbol \lambda)
\end{smallmatrix}\right]+\left[\begin{smallmatrix}
0 & 0 & {\bf{A}}^{\top} \\ 
0 & 0 & {\bf{L}}\\
-{\bf{A}} & -{\bf{L}} & 0
\end{smallmatrix}\right]\left[\begin{smallmatrix}
\boldsymbol x \\ 
\boldsymbol z\\
\boldsymbol \lambda
\end{smallmatrix}\right].
\end{aligned}
\end{equation}

Analogously to $\hat{\mc A}$, we replace $\FF$ in (\ref{eq_grad}) with the approximation $F^\textup{SA}$ in \eqref{eq_F_SA} and $\mc C$ with $\hat{\mc C}=\mc C^\textup{SA}$.
Then, we can write Algorithm \ref{algo_si} in compact form as 
\begin{equation}\label{eq_algo}
\bs\omega^{k+1}=(\op{Id}+\Psi^{-1}\mc D)^{-1}(\bs\omega^k-\Psi^{-1}\hat{\mc C}(2\bs\omega^k-\bs\omega^{k-1}))
\end{equation}
where $\Psi$ is the preconditioning matrix. Specifically, let $\alpha^{-1}=\op{diag}\{\alpha_1^{-1}\op{I}_{n_1},\dots,\alpha_N^{-1}\op{I}_{n_N}\}\in\RR^{n\times n}$ and similarly $\sigma^{-1}$ and $\nu^{-1}$ of suitable dimensions. Then, we have
\begin{equation}\label{phi}
\Psi=\left[\begin{smallmatrix}
\alpha^{-1} & 0 & -{\bf{A}}^\top\\
0 & \nu^{-1} & -{\bf{L}}\\
-{\bf{A}} & -{\bf{L}} & \sigma^{-1}
\end{smallmatrix}\right].
\end{equation}
By expanding \eqref{eq_algo} with $\hat{\mc C}$ and $\mc D$ as in \eqref{eq_CD} and $\Psi$ as in \eqref{phi}, we obtain the iterations in Algorithm \ref{algo_si}. 

We note that in general the extended operators in \eqref{eq_AB} and \eqref{eq_CD} have different monotonicity properties. Which one specifically is discussed for the convergence analysis (Section \ref{sec_conv_uni} and \ref{sec_conv_coco}, respectively) that follows from the considerations in the next section.

\vspace{-.1cm}
\section{Technical discussion on weak sharpness and cocoercivity}\label{sec_coco}
The original proof of the SPRG presented in \cite{cui2019} for SVI shows convergence under the assumption of monotonicity and weak sharpness. The weak sharpness property was first introduced to characterize the minima of
\begin{equation}\label{eq_opt}
\min_{x\in\mc X}f(x)
\end{equation}
with $f:\mc X\to\bar\RR$ \cite{burke1993}. It was presented as en extension of the concept of strong (or sharp) solution, i.e., for all $x^*\in\mc X^*=\op{SOL}(f,\mc X)$
$f(x)\geq f(x^*)+\rho\norm{x-x^*},$
which holds if there is only one minimum.
For generalizing non-unique solutions, the following definition was proposed in \cite{burke1993}: a set $\mc X^*$ is a set of weak sharp minima for the function $f$ if, for all $x\in\mc X$ and $x^*\in\mc X^*$,
\begin{equation}\label{eq_weak_f}
f(x)\geq f(x^*)+\rho \op{dist}(x,\mc X^*)
\end{equation}
where $\op{dist}(x,\mc X^*)=\inf_{x^*\in\mc X^*}\norm{x-x^*}$. We note that a strong solution is also a weak sharp minimum while the contrary holds only if the solution is unique \cite{burke1993}.

The concept was later extended to variational inequalities in \cite{marcotte1998}, using the formal definition 
\begin{equation}\label{eq_tangent}
\textstyle{-\FF(\bs x^{*}) \in \operatorname{int}\left(\bigcap\nolimits_{\bs x \in \bs{\mc{X}}^{*}}[\op{T}_{\bs{\mc X}}(\bs x) \cap \op{N}_{\bs{\mc{X}}^{*}}(\bs x)]^{\circ}\right),}
\end{equation}
which was already proved to be equivalent to \eqref{eq_weak_f} for the problem in $\eqref{eq_opt}$ when $\FF(x^*)=\nabla f(x^*)$.
Unfortunately, the characterization in \eqref{eq_tangent} is hard to use in a convergence proof. Hence, more practical conditions have been proposed. The first one \cite{marcotte1998} relies on the gap function $G$ and it reads as
\begin{equation}\label{eq_weak_gap}
G(x)=\max_{y\in\mc X}\langle\FF(y),x-y\rangle\geq \rho \op{dist}(x,\mc X^*).
\end{equation}
Another condition, used in the convergence proof of the SPRG \cite{cui2019}, was proposed in \cite{yousefian2014}: 
\begin{equation}\label{eq_weak}
\left\langle \FF(\bs x^{*}), \bs x-\bs x^{*}\right\rangle \geq \rho\; \mathrm{dist}(\bs x, \bs{\mc X}^{*}),
\end{equation}
for all $\bs x^{*} \in \bs {\mc X}^{*}$ and $\bs x \in \bs{\mc X}$.
For the weak sharpness definition in \eqref{eq_tangent} to be equivalent to \eqref{eq_weak_gap} and \eqref{eq_weak}, the (pseudogradient) mapping should have the \textit{F-unique} property, i.e., $\FF(\op{SOL}(\FF,\bs{\mc X}))$ should be at most a singleton \cite[Section 2.3.1]{facchinei2007}. The class of operators that certainly have this property is that of \textit{monotone$^+$} operators, namely, a monotone mapping $F$ such that for all $\bs x,\bs y\in\bs{\mc X}$
$
\langle F(\bs y)-F(\bs x), \bs y-\bs x\rangle= 0 \Rightarrow F(\bs y)=F(\bs x).
$
If a mapping is monotone$^+$, then \eqref{eq_tangent} is equivalent to \eqref{eq_weak_gap} and \eqref{eq_weak} \cite[Thm. 4.1]{marcotte1998}. 
The monotone$^+$ property does not necessarily hold for the extended operator $\mc C$ in \eqref{eq_CD}, even if it holds for $\FF$. However, it holds if the operator is cocoercive \cite[Def. 2.3.9]{facchinei2007}.
We conclude this section with some examples showing that the condition in \eqref{eq_weak} may hold also if the mapping is not monotone$^+$ and that the domains are relevant for the validity of the assumption. For more details on monotone$^+$ operators and the weak sharpness property, we refer to \cite{iusem1998,crouzeix2000,marcotte1998,hu2011}. 
\begin{example}\cite{marcotte1998}
Consider the variational inequality in \eqref{eq_vi} where $F(\bs x) = \op{col}(-x_2,2x_1)$ and $\bs{\mc X} = [0,1]^2$. The mapping $F$ is pseudomonotone but not monotone$^+$ on $\bs{\mc X}$. The solution set is $\bs{\mc X}^*= \{\bs x\in \bs{\mc X} : x_2 = 0\}$ and it holds that
\begin{equation}\label{eq_example}
\begin{aligned}
G(\bs x) &=\max _{\bs y \in \bs{\mc X}}\langle F(\bs y),\bs x-\bs y\rangle=\max _{\bs y \in \bs{\mc X}}-x_{1} y_{2}-y_{1} y_{2}+2 x_{2} y_{1} \\
&=2 x_{2}=2 \operatorname{dist}(\bs x, \bs{\mc X}^{*}).
\end{aligned}
\end{equation}
Therefore, $\bs{\mc X}^{*}$ satisfies \eqref{eq_weak_gap} with $\rho=2$ but, for any $\bs x^{*}\in\bs{\mc X}^{*}$, $[T_{\bs{\mc X}}(\bs x^{*}) \cap N_{\bs{\mc X}^*}(\bs x^{*})]^{\circ}=\{x_{2}^{*} \leq 0\}$ and $-F(\bs x^{*})\notin\bigcap_{\bs x^{*} \in \bs{\mc X}}[T_{\bs{\mc X}}(\bs x^{*}) \cap N_{\bs{\mc X}^{*}}(\bs x^{*})]^{\circ}$. Thus, the solution set $\bs{\mc X}^{*}$ is not weakly sharp.\fineass
\end{example}

\begin{example}\label{ex_academic}
Consider the mapping $F(\bs x)=\op{col}(-x_2,x_1)$ and the associate variational inequality in \eqref{eq_vi} with $\bs{\mc X} = [0,1]^2$. Then the mapping $F$ is monotone but not monotone$^+$ on $\bs{\mc X}$. The solution set is $\bs{\mc X}^*= \{\bs x\in \bs{\mc X} : x_2 = 0\}$ and, similarly to \eqref{eq_example}, the conditions \eqref{eq_weak_gap} and \eqref{eq_weak} hold.

Now, let $\bs{\mc X}=\RR^2$. In this case, there is only one solution and $\bs{\mc X}^*=\{\bs 0_2\}$. However, \eqref{eq_weak} reads as
$\langle F(\bs 0),\bs x\rangle=0\geq\rho \;\op{dist}(\bs x,\mc X^*)=\norm{x},$
which is false.\fineass
\end{example}
\vspace{-.1cm}
\section{Convergence under uniqueness of solution}\label{sec_conv_uni}
In light of the considerations in Section \ref{sec_coco}, we know that a unique solution is also a weak solution and that \eqref{eq_weak} may hold even if the mapping is not monotone$^+$. Therefore, here we consider the case of merely monotone operators but with unique solution and prove that the proposed (non-preconditioned) Algorithm \ref{algo_no} converges to a v-SGNE.

First, the following lemma ensure that the zeros of $\mc A+\mc B$ are v-SGNEs.

\begin{lemma}\label{lemma_zeri}
Let Assumptions \ref{ass_J}-\ref{ass_graph} hold. Consider the operators $\mc T$ in (\ref{eq_T}) and $\mc A$ and $\mc B$ in (\ref{eq_AB}). Then:
\begin{itemize}
\item[(i)] If $\boldsymbol\omega^* \in \op{zer}(\mc A+\mc B)$, then $\boldsymbol{x}^{*}$ is a v-SGNE of game in \eqref{eq_game}, i.e., $\boldsymbol x^*$ solves the $\op{SVI} (\bs{\mc X}, \FF)$ in (\ref{eq_vi}). Moreover $\boldsymbol \lambda^{*}=\mathbf{1}_{N} \otimes \lambda^{*},$ and $(\boldsymbol x^*,\bs\lambda^{*})$ satisfy the KKT condition in (\ref{game_kkt}) i.e., $\op{col}(\boldsymbol x^*,\bs \lambda^{*}) \in \op{zer}(\mc T)$
\item[(ii)] $\op{zer}(\mc T) \neq \emptyset$ and $\op{zer}(\mc A+\mc B) \neq \emptyset$
\end{itemize}
\end{lemma}
\begin{proof}
It follows from \cite[Th. 2]{yi2019}.
\end{proof}

To ensure that $\mc A$ and $\mc B$ have the properties for the convergence result, we make the following assumption.
\begin{assumption}\label{ass_mono}
$\FF$ as in \eqref{eq_grad} is monotone and $\ell_\FF$-Lipschitz continuous for some $\ell_\FF>0$.
\fineass\end{assumption}
Then, the following holds.

\begin{lemma}\label{lemma_op}
Let Assumptions \ref{ass_J} and \ref{ass_mono} hold and let $\Phi\succ 0$. Then, $\mc A$ and $\mc B$ in \eqref{eq_AB} have the following properties.
\begin{enumerate}
\item[(i)] $\mc A$ is monotone and $\ell_{\mc A}$-Lipschitz continuous.
\item[(ii)] The operator $\mc B$ is maximally monotone.
\item[(iii)] $\Phi^{-1}\mc A$ is monotone and $\ell_{\Phi}$-Lipschitz continuous.
\item[(iv)] $\Phi^{-1}\mc B$ is maximally monotone.
\end{enumerate}
\end{lemma}
\begin{proof}
It follows similarly to \cite{yi2019,franci2020tac}.
\end{proof}

To guarantee that the weak sharpness property holds, we assume to have a strong solution.

\begin{assumption}\label{ass_sol}
The SVI in \eqref{eq_vi} has a unique solution.
\fineass\end{assumption}

We can now state the convergence result.
\begin{theorem}\label{theo_conv}
Let Assumptions \ref{ass_J}-\ref{ass_sol} hold. Then, the sequence $(\bs x_k)_{k\in\NN}$ generated by Algorithm \ref{algo_no} with $F^{\textup{SA}}$ as in \eqref{eq_F_SA} converges a.s. to a v-SGNE of the game in \eqref{eq_game}.
\end{theorem}
\begin{proof}
The iterations of Algorithm \ref{algo_no} are obtained by expanding (\ref{eq_algo2}) and solving for $\bs x_{k}$, $\bs z_{k}$ and $\bs \lambda_{k}$. Therefore, Algorithm \ref{algo_no} is a SPRG iteration as in (\ref{eq_algo2}). The convergence of the sequence $(\boldsymbol x^k,\boldsymbol\lambda^k)$ to a v-GNE of the game in \eqref{eq_game} then follows by \cite[Prop. 10]{cui2019} and Lemma \ref{lemma_zeri} since $\Phi^{-1}\mc A$ is monotone by Lemma \ref{lemma_op} and has a unique solution.
\end{proof}

\vspace{-.1cm}
\section{Convergence under cocoercivity}\label{sec_conv_coco}
We know from Section \ref{sec_coco} that cocoercive operators have the weak sharpness property. Here we consider this case.

\begin{assumption}\label{ass_coco}
$\FF$ is $\beta$-cocoercive for some $\beta>0$.
\fineass
\end{assumption}
\begin{remark}
If a function is $\beta$-cocoercive, it is also $1/\beta$-Lipschitz continuous \cite[Remark 4.15]{bau2011}.\fineass
\end{remark}
We note that the operator $\mc A$ in \eqref{eq_AB} contains a skew symmetric matrix that is not cocoercive. For this reason we consider the splitting in \eqref{eq_CD}. While Lemma \ref{lemma_zeri} guarantees that the zeros of $\mc C+\mc D$ are the same as the zeros of $\mc T$ in \eqref{eq_T}, we now show the necessary monotonicity properties of the extended operators $\mc C$ and $\mc D$ in \eqref{eq_CD}. 
\begin{lemma}\label{propertiesAB}
Let Assumptions \ref{ass_J} and \ref{ass_coco} hold and let $\Psi\succ0$. Then, $\mc C$ and $\mc D$ in (\ref{eq_CD}) have the following properties.
\begin{itemize}
\item[(i)] $\mc C$ is $\theta$-cocoercive where $0<\theta \leq\min \left\{\frac{1}{2 d^{*}}, \beta\right\}$ and $d^*$ is the maximum weighted degree of $\mc G^\lambda$;
\item[(ii)] $\mc D$ is maximally monotone;
\item[(iii)] $\Psi^{-1}\mc C$ is $\theta\tau$-cocoercive, with $\tau=\frac{1}{|\Psi^{-1}|}$;
\item[(iv)] $\Psi^{-1}\mc D$ is maximally monotone. 
\end{itemize}
\end{lemma}
\begin{proof}
It follows from \cite[Lem. 5]{yi2019} and \cite[Lem. 7]{yi2019}.
\end{proof}

Furthermore, since the preconditioning matrix must be positive definite, we postulate the following assumption on the step sizes \cite{yi2019}. 
\begin{assumption}\label{ass_phi}
Let $\theta$ be the cocoercivity constant as in Lemma \ref{propertiesAB}, $\tau=\frac{1}{|\Psi^{-1}|}\in(0,\frac{\theta}{8})$ and the step sizes $\alpha$, $\nu$ and $\sigma$ satisfy, for all $i\in\mc I$,
\begin{equation*}\label{parameters_phi}
\begin{aligned}
0&<\alpha_{i} \leq\hspace{-.1cm}\textstyle{(\max _{j=1, \ldots, n_{i}}\{\sum\nolimits_{k=1}^{m}|[A_{i}^{T}]_{j k}|\}+\tau)^{-1}} \\ 
0&<\nu_{i} \leq(2 d_{i}+\tau)^{-1}\\
0&<\sigma_{i} \leq\hspace{-.1cm}\textstyle{(\max _{j=1, \ldots, m}\{\sum\nolimits_{k=1}^{k_{i}}|[A_{i}]_{j k}|\}+2 d_{i}+\tau)^{-1}}\\
\end{aligned}
\end{equation*}
where $[A_i^\top]_{jk}$ indicates the entry $(j,k)$ of the matrix $A_i^\top$. \fineass
\end{assumption}

We are now ready to state our convergence result.
\begin{theorem}\label{theo_conv}
Let Assumptions \ref{ass_J}-\ref{ass_graph} and \ref{ass_coco}-\ref{ass_phi} hold. Then, the sequence $(\bs x_k)_{k\in\NN}$ generated by Algorithm \ref{algo_si} with $F^{\textup{SA}}$ as in \eqref{eq_F_SA} converges a.s. to a v-SGNE of the game in \eqref{eq_game}.
\end{theorem}
\begin{proof}
The iterations of Algorithm \ref{algo_si} are obtained by expanding (\ref{eq_algo}) and solving for $\bs x_{k}$, $\bs z_{k}$ and $\bs \lambda_{k}$. Therefore, Algorithm \ref{algo_si} is a SPRG iteration as in (\ref{eq_algo}). The convergence of the sequence $(\boldsymbol x^k,\boldsymbol\lambda^k)$ to a v-GNE of the game in \eqref{eq_game} then follows by \cite[Prop. 10]{cui2019} and Lemma \ref{lemma_zeri} since $\Psi^{-1}\mc C$ is cocoercive by Lemma \ref{propertiesAB}.
\end{proof}

\vspace{-.1cm}
\section{Numerical Simulations}\label{sec_sim}

Let us propose some numerical evaluations to validate the analysis with a realistic application to an electricity market with capacity constraints, casted as a network Cournot game with markets capacity constraints \cite{yi2019,yu2017,cui2019}. 

We consider a set of $N=20$ companies selling energy to $m=7$ markets. Each generator decides the quantity of energy $x_i$ to deliver to the $n_i$ markets it is connected with. Each company $i$ has a a production limit of the form $0 < x_i < \gamma_i$ where each component of $\gamma_i$ is randomly drawn from $[1, 1.5]$. Each company has a cost of production $c_i(x_i) = 1.5 x_i+q_i$, where $q_i$ is a given constant. 
Each market $j$ has a bounded capacity $b_j$, randomly drawn from $[0.5, 1]$. 
The collective constraints are then given by $A\boldsymbol x\leq b$ where $A=[A_1,\dots,A_N]$.
The prices of the markets are collected in $P:\RR^m\times \Xi\to\RR^m$. The uncertainty variable $\xi$, which represents the demand uncertainty, appears in this functional. $P$ is supposed to be a linear function and reads as $P(\xi) = \bar P(\xi)-DA\boldsymbol x$. Each component of $\bar P =\op{col}(\bar P_1(\xi),\dots,\bar P_7(\xi))$ is taken with a normal distribution with mean $3$ and finite variance. The entries of $D$ are randomly taken in $[0.5,1]$. 
The cost function of each agent is given by
$\JJ_i(x_i,x_{-i},\xi)=c_i(x_i)-\EE[P(\xi)^\top(A\boldsymbol x)A_ix_i].$
and it is influenced by the variables of the companies selling in the same market as in \cite[Fig. 1]{yi2019}. The dual variables graph is a cycle graph with the addiction of the edges $(2,15)$ and $(6,13)$ \cite{yi2019}. 

We simulate the SpFB, the forward-backward-forward (SFBF) and the extragradient (SEG) algorithms to make a comparison with our SPRG and SpPRG, using the SA scheme. The parameters $\alpha$, $\nu$ and $\sigma$ are taken to be the highest possible that guarantee convergence. 
As a measure of the distance from the solution, we consider the residual, $\op{res}(x^k)=\norm{x^k-\op{proj}_{C}(x^k-F(x^k))}$, which is equal zero if and only if $x$ is a solution.
The plot in Fig. \ref{plot_sol} shows how the residual varies in the number of iterations. The performances of SpPRG and SPRG are very similar and the difference in the trajectory is related to the different step sizes which depend on the Lipschitz constant of $\mc C$ in \eqref{eq_CD} and $\mc A$ in \eqref{eq_AB} respectively.


\begin{figure}[t]
\centering
\includegraphics[width=.9\columnwidth]{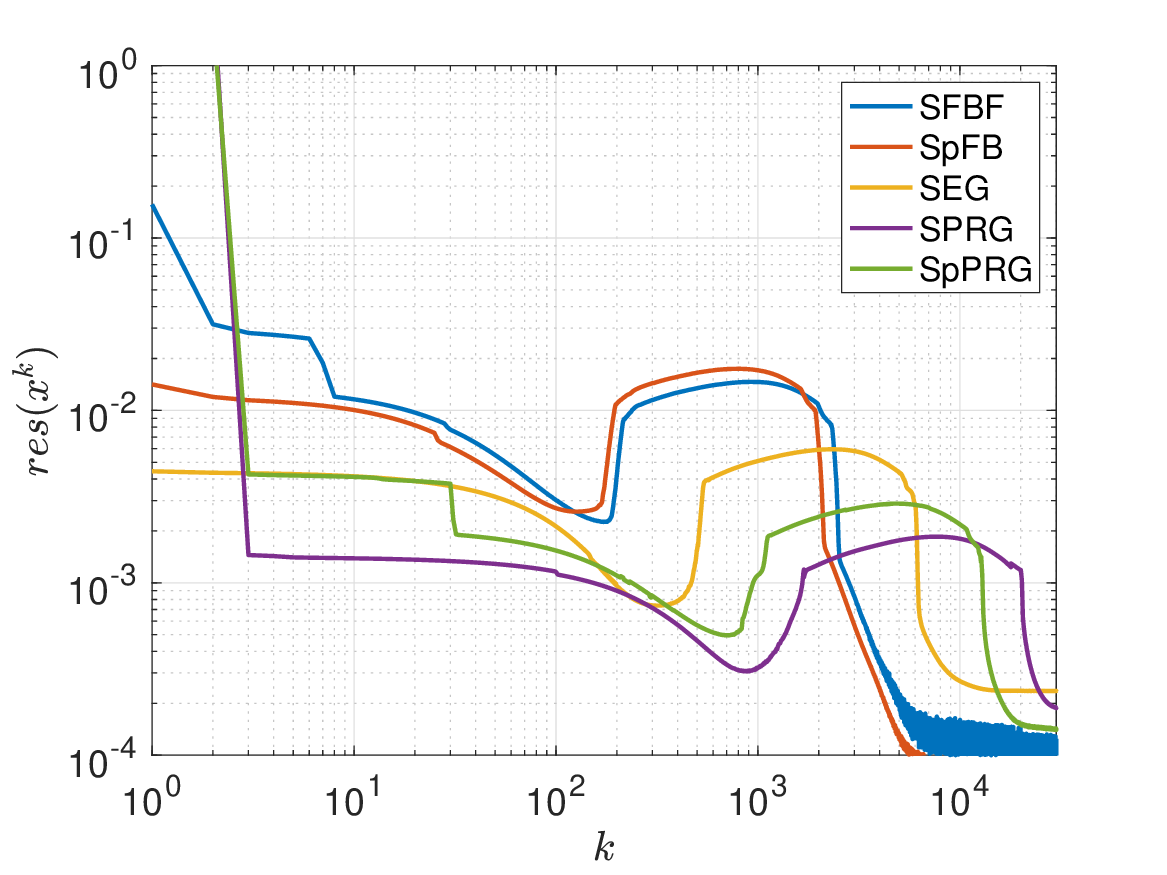}
\caption{Relative distance of the primal variable from the solution.}\label{plot_sol}
\end{figure}
\vspace{-.1cm}

\section{Conclusion}

The stochastic projected reflected gradient algorithm is applicable to distributed stochastic generalized Nash equilibrium seeking. To guarantee convergence to a solution and to obtain a distributed algorithm, preconditioning may be used and the pseudogradient mapping should be cocoercive. However, should the equilibrium be unique, the cocoercivity assumption can be reduced to mere monotonicity. 

\vspace{-.1cm}

\bibliographystyle{IEEEtran}
\bibliography{Biblio}

\begin{thebibliography}{10}
\providecommand{\url}[1]{#1}
\csname url@samestyle\endcsname
\providecommand{\newblock}{\relax}
\providecommand{\bibinfo}[2]{#2}
\providecommand{\BIBentrySTDinterwordspacing}{\spaceskip=0pt\relax}
\providecommand{\BIBentryALTinterwordstretchfactor}{4}
\providecommand{\BIBentryALTinterwordspacing}{\spaceskip=\fontdimen2\font plus
\BIBentryALTinterwordstretchfactor\fontdimen3\font minus
  \fontdimen4\font\relax}
\providecommand{\BIBforeignlanguage}[2]{{%
\expandafter\ifx\csname l@#1\endcsname\relax
\typeout{** WARNING: IEEEtran.bst: No hyphenation pattern has been}%
\typeout{** loaded for the language `#1'. Using the pattern for}%
\typeout{** the default language instead.}%
\else
\language=\csname l@#1\endcsname
\fi
#2}}
\providecommand{\BIBdecl}{\relax}
\BIBdecl

\bibitem{koshal2013}
J.~Koshal, A.~Nedic, and U.~V. Shanbhag, ``Regularized iterative stochastic
  approximation methods for stochastic variational inequality problems,''
  \emph{IEEE Transactions on Automatic Control}, vol.~58, no.~3, pp. 594--609,
  2013.

\bibitem{yu2017}
C.-K. Yu, M.~Van Der~Schaar, and A.~H. Sayed, ``Distributed learning for
  stochastic generalized {Nash} equilibrium problems,'' \emph{IEEE Transactions
  on Signal Processing}, vol.~65, no.~15, pp. 3893--3908, 2017.

\bibitem{franci2020tac}
B.~Franci and S.~Grammatico, ``A distributed forward-backward algorithm for
  stochastic generalized {Nash} equilibrium seeking,'' \emph{IEEE Transactions
  on Automatic Control}, 2020.

\bibitem{abada2013}
I.~Abada, S.~Gabriel, V.~Briat, and O.~Massol, ``A generalized
  {Nash}--{Cournot} model for the northwestern european natural gas markets
  with a fuel substitution demand function: The gammes model,'' \emph{Networks
  and Spatial Economics}, vol.~13, no.~1, pp. 1--42, 2013.

\bibitem{watling2006}
D.~Watling, ``User equilibrium traffic network assignment with stochastic
  travel times and late arrival penalty,'' \emph{European Journal of
  Operational Research}, vol. 175, no.~3, pp. 1539--1556, 2006.

\bibitem{henrion2007}
R.~Henrion and W.~R{\"o}misch, ``On m-stationary points for a stochastic
  equilibrium problem under equilibrium constraints in electricity spot market
  modeling,'' \emph{Applications of Mathematics}, vol.~52, no.~6, pp. 473--494,
  2007.

\bibitem{yi2019}
P.~Yi and L.~Pavel, ``An operator splitting approach for distributed
  generalized {Nash} equilibria computation,'' \emph{Automatica}, vol. 102, pp.
  111--121, 2019.

\bibitem{belgioioso2018}
G.~Belgioioso and S.~Grammatico, ``Projected-gradient algorithms for
  generalized equilibrium seeking in aggregative games are preconditioned
  forward-backward methods,'' in \emph{2018 European Control Conference
  (ECC)}.\hskip 1em plus 0.5em minus 0.4em\relax IEEE, 2018, pp. 2188--2193.

\bibitem{robbins1951}
H.~Robbins and S.~Monro, ``A stochastic approximation method,'' \emph{The
  annals of mathematical statistics}, pp. 400--407, 1951.

\bibitem{cui2016}
S.~Cui and U.~V. Shanbhag, ``On the analysis of reflected gradient and
  splitting methods for monotone stochastic variational inequality problems,''
  in \emph{2016 IEEE 55th Conference on Decision and Control (CDC)}.\hskip 1em
  plus 0.5em minus 0.4em\relax IEEE, 2016, pp. 4510--4515.

\bibitem{staudigl2019}
R.~Bot, P.~Mertikopoulos, M.~Staudigl, and P.~Vuong, ``Mini-batch
  forward-backward-forward methods for solving stochastic variational
  inequalities,'' \emph{Stochastic Systems}, 2020.

\bibitem{iusem2017}
A.~Iusem, A.~Jofr{\'e}, R.~I. Oliveira, and P.~Thompson, ``Extragradient method
  with variance reduction for stochastic variational inequalities,'' \emph{SIAM
  Journal on Optimization}, vol.~27, no.~2, pp. 686--724, 2017.

\bibitem{cui2019}
S.~Cui and U.~V. Shanbhag, ``On the analysis of variance-reduced and randomized
  projection variants of single projection schemes for monotone stochastic
  variational inequality problems,'' \emph{arXiv preprint arXiv:1904.11076},
  2019.

\bibitem{malitsky2015}
Y.~Malitsky, ``Projected reflected gradient methods for monotone variational
  inequalities,'' \emph{SIAM Journal on Optimization}, vol.~25, no.~1, pp.
  502--520, 2015.

\bibitem{bau2011}
H.~H. Bauschke, P.~L. Combettes \emph{et~al.}, \emph{Convex analysis and
  monotone operator theory in Hilbert spaces}.\hskip 1em plus 0.5em minus
  0.4em\relax Springer, 2011, vol. 408.

\bibitem{facchinei2007}
F.~Facchinei and J.-S. Pang, \emph{Finite-dimensional variational inequalities
  and complementarity problems}.\hskip 1em plus 0.5em minus 0.4em\relax
  Springer Science \& Business Media, 2007.

\bibitem{ravat2011}
U.~Ravat and U.~V. Shanbhag, ``On the characterization of solution sets of
  smooth and nonsmooth convex stochastic {Nash} games,'' \emph{SIAM Journal on
  Optimization}, vol.~21, no.~3, pp. 1168--1199, 2011.

\bibitem{facchinei2007vi}
F.~Facchinei, A.~Fischer, and V.~Piccialli, ``On generalized {Nash} games and
  variational inequalities,'' \emph{Operations Research Letters}, vol.~35,
  no.~2, pp. 159--164, 2007.

\bibitem{burke1993}
J.~V. Burke and M.~C. Ferris, ``Weak sharp minima in mathematical
  programming,'' \emph{SIAM Journal on Control and Optimization}, vol.~31,
  no.~5, pp. 1340--1359, 1993.

\bibitem{marcotte1998}
P.~Marcotte and D.~Zhu, ``Weak sharp solutions of variational inequalities,''
  \emph{SIAM Journal on Optimization}, vol.~9, no.~1, pp. 179--189, 1998.

\bibitem{yousefian2014}
F.~Yousefian, A.~Nedi{\'c}, and U.~V. Shanbhag, ``Optimal robust smoothing
  extragradient algorithms for stochastic variational inequality problems,'' in
  \emph{53rd IEEE Conference on Decision and Control}.\hskip 1em plus 0.5em
  minus 0.4em\relax IEEE, 2014, pp. 5831--5836.

\bibitem{iusem1998}
A.~N. Iusem, ``On some properties of paramonotone operators,'' \emph{Journal of
  Convex Analysis}, vol.~5, pp. 269--278, 1998.

\bibitem{crouzeix2000}
J.-P. Crouzeix, P.~Marcotte, and D.~Zhu, ``Conditions ensuring the
  applicability of cutting-plane methods for solving variational
  inequalities,'' \emph{Mathematical Programming}, vol.~88, no.~3, pp.
  521--539, 2000.

\bibitem{hu2011}
Y.~Hu and W.~Song, ``Weak sharp solutions for variational inequalities in
  banach spaces,'' \emph{Journal of mathematical analysis and applications},
  vol. 374, no.~1, pp. 118--132, 2011.

\end{thebibliography}

\end{document}